\documentclass{article}
\usepackage[english]{babel}
\usepackage[utf8]{inputenc}
\usepackage{amsmath}
\usepackage{graphicx}
\usepackage{amsfonts}
\usepackage{enumitem}
\usepackage{amssymb}
\usepackage[colorinlistoftodos]{todonotes}
\usepackage{geometry}
\usepackage{amsthm}
\newcommand{\abs}[1]{\left\lvert#1\right\rvert}

\newcommand{\paren}[1]{\left( #1 \right)}

\newcommand{\set}[1]{\left\{ #1 \right\}}
\newcommand{\setcond}[2]{\left\{ #1 \;\middle\vert\; #2 \right\}}

\newcommand{\RR}{\mathbb{R}}

\newcommand{\ZZ}{\mathbb{Z}}
\newcommand{\QQ}{\mathbb{Q}}

\newtheorem{thm}{Theorem}[section]
\newtheorem{lem}[thm]{Lemma}

\theoremstyle{definition}

\title{Sums of transcendental dilates}
\author{David Conlon\thanks{Department of Mathematics, Caltech, Pasadena, CA 91125, USA. Email: {\tt dconlon@caltech.edu}. Research supported by NSF Award DMS-2054452.} \and Jeck Lim\thanks{Department of Mathematics, Caltech, Pasadena, CA 91125, USA. Email: {\tt jlim@caltech.edu}. Research partially supported by an NUS Overseas Graduate Scholarship.}}
\date{}

\begin{document}
\maketitle

\begin{abstract}
We show that there is an absolute constant $c>0$ such that 
$|A+\lambda\cdot A|\geq e^{c\sqrt{\log |A|}}|A|$
for any finite subset $A$ of $\mathbb{R}$ and any transcendental number $\lambda\in\mathbb{R}$. By a construction of Konyagin and \L aba, this is best possible up to the constant $c$.
\end{abstract}

\section{Introduction}

For any subset $A$ of $\mathbb{R}$ and any $\lambda \in \mathbb{R}$, let 
\[A + \lambda \cdot A = \{a + \lambda a' : a, a' \in A\}.\]
Our interest here will be in estimating the minimum size of such sums of dilates 
given $|A|$. 

When $\lambda$ is rational, say $\lambda = p/q$ with $p$ and $q$ coprime, a result of Bukh~\cite{B08} implies that 
$$|A+ \lambda\cdot A| \geq (|p| + |q|)|A| - o(|A|),$$
which is best possible up to the lower-order term (though see~\cite{BS14} for an improvement of the lower-order term to a constant depending only on $\lambda$). 
The more general case where $\lambda$ is algebraic
has also been studied in some depth. In particular, 
a result of the authors~\cite{CL22} says that if $\lambda = (p/q)^{1/d}$ for some $p, q, d \in \mathbb{N}$, each taken as small as possible for such a representation, then
$$|A+ \lambda\cdot A| \geq (p^{1/d} + q^{1/d})^d|A| - o(|A|),$$
which is again best possible up to the lower-order term. Moreover, as noted by Krachun and Petrov~\cite{KP20}, for any fixed algebraic number $\lambda$, the minimum size of $|A+ \lambda\cdot A|$ is always at most linear in $|A|$. 

For $\lambda$ transcendental, the picture is very different. Indeed, Konyagin and \L aba~\cite{KL06} showed that in this case there exists an absolute constant $c > 0$ such that 
$$|A + \lambda \cdot A| \geq c \frac{\log |A|}{\log \log |A|} |A|.$$
That is, $|A + \lambda \cdot A|$ can no longer be linear in $|A|$. This result was subsequently improved by Sanders~\cite{S08}, by Schoen~\cite{Sch11} and again by Sanders~\cite{S12} using successive quantitative refinements of Freiman's theorem~\cite{Fr73} on sets of small doubling, with Sanders' second bound saying that there exists an absolute constant $c > 0$ such that, for $|A|$ sufficiently large, 
$$|A + \lambda \cdot A| \geq e^{\log^c |A|} |A|.$$
This already comes quite close to matching the best known upper bound, due to Konyagin and \L aba~\cite{KL06}, which says that there exists $c' > 0$ and, for any fixed transcendental number $\lambda$,  arbitrarily large finite subsets $A$ of $\RR$ such that 
$$|A + \lambda \cdot A| \leq e^{c' \sqrt{\log |A|}} |A|.$$
Our main result says that this upper bound is in fact best possible up to the constant $c'$.

\begin{thm} \label{thm:mainintro}
There is an absolute constant $c>0$ such that 
$$|A+\lambda\cdot A|\geq e^{c\sqrt{\log |A|}}|A|$$
for any finite subset $A$ of $\mathbb{R}$ and any transcendental number $\lambda\in\mathbb{R}$.
\end{thm}

Before proceeding to the proof of this theorem, let us briefly look at the upper bound, which comes from considering sets of the form
\[A = \left\{\sum_{i=1}^m a_i \lambda^i: (a_1, \dots, a_m) \in [n]^m\right\}.\]
This set has size $n^m$ and 
\[A + \lambda \cdot A \subset \left\{\sum_{i=1}^{m+1} b_i \lambda^i : (b_1, \dots, b_{m+1}) \in [2n]^{m+1}\right\},\]
which has size $(2n)^{m+1}$. If we take $n = 2^m$, we have $|A| = n^m = 2^{m^2}$, so that 
$$|A+\lambda\cdot A| \leq (2n)^{m+1} = 2^{(m+1)^2} \leq e^{c' \sqrt{\log |A|}} |A|$$
for some $c' > 0$, as required. This bound is reminiscent, both in its form and its proof, of Behrend's lower bound~\cite{B46} for the largest subset of $[n]$ containing no three-term arithmetic progressions. 
Our Theorem~\ref{thm:mainintro} is arguably the first example where such a bound is known to be tight to this level of accuracy. 

\section{Proof of Theorem~\ref{thm:mainintro}}

To begin, we use a simple observation of Krachun and Petrov to recast the problem. 

\begin{lem}[Krachun--Petrov~\cite{KP20}] \label{lem:alg}
Suppose that $\lambda\in\mathbb{C}$ and $A$ is a finite subset of $\mathbb{C}$. Then there exists $B\subset\QQ[\lambda]$ such that $|B| = |A|$ and $|B + \lambda \cdot B| \leq |A +\lambda \cdot A|$.
\end{lem}

Suppose now that $V$ is the $\QQ$-vector space $\QQ[\lambda]$ with basis $\set{1,\lambda,\lambda^2,\ldots}$. For any positive integer $d$, let $V_d\subset V$ be the $d$-dimensional subspace spanned by $\set{1,\lambda,\lambda^2,\ldots,\lambda^{d-1}}$, noting that $V=\bigcup_d V_d$. For any finite $A \subset V$, we must have $A\subset V_d$ for some $d$. Multiplication by $\lambda$ therefore corresponds to taking the linear map $\Phi:V\to V$ given by the union of the maps $V_d\to V_{d+1}$ with
$$(x_1,\ldots,x_d)\mapsto (0, x_1,\ldots,x_d).$$
Thus, the problem of estimating $|A + \lambda \cdot A|$ for finite $A \subset \RR$ and $\lambda$ transcendental is equivalent to estimating $|A + \Phi(A)|$ for finite $A \subset V$. In particular, we may reformulate Theorem~\ref{thm:mainintro} in the following terms.

\begin{thm} \label{thm:main}
There is an absolute constant $c>0$ such that if $A\subset V$ with $|A|=n$, then
$$|A+\Phi(A)|\geq e^{c\sqrt{\log n}}n.$$
\end{thm}

We will focus on proving this latter result, which bears some relation to our recent work~\cite{CL22} on sums of linear transformations, from here on.

Before getting to the proof proper, we first note a few additional results that we will need. The first is a 
discrete variant of the Brunn--Minkowski theorem taken from~\cite{CL22}. In what follows, for each $I\subseteq [d]$, we write $p_I:\RR^d\to \RR^d$ for the projection onto the coordinates indexed by $I$, setting all other coordinates to 0. Note that we may naturally extend the definition of $p_I$ to $V_d$, and hence to $V$, by identifying $V_d$ with $\QQ^d$.

\begin{lem}[Conlon--Lim~{\cite[Lemma~2.1]{CL22}}] \label{lem:discbm}
For any finite subsets  $A,B$  of $\RR^d$,
$$\sum_{I\subseteq [d]}|p_I(A+B)|\geq (|A|^{1/d}+|B|^{1/d})^d.$$
\end{lem}

For our next result, we need the following estimate of Ruzsa~\cite{R94} for the size of sumsets in $\mathbb{R}^d$. We say that a subset $C$ of $\mathbb{R}^d$
is $k$-dimensional and write $\dim(C) = k$ if the dimension of the affine subspace spanned by $C$ is $k$. 

\begin{lem}[Ruzsa~\cite{R94}] \label{lem:hdsums}
If $A,B\subset \RR^d$, $|A|\geq |B|$ and $\dim(A+B)=d$, then 
$$|A+B|\geq |A|+d|B|-\frac{d(d+1)}{2}.$$
\end{lem}

For $a \in V$, write $p_k(a)$ for the vector obtained by removing the $k$-th coordinate from $a$. For  $A \subset V$ and $x \in p_k(A)$, let $A_x = p_k^{-1}(x)$. We define the compression $C_k(A)$ of $A$ along the $k$-th coordinate to be the set $A'$ such that $p_k(A') = p_k(A)$ and, for each $x \in p_k(A)$, the $k$-th coordinates of $A'_x$ are $0, 1, \dots, |A_x| - 1$. 
It is known (see, for example, \cite[Lemma~2.1]{CL22}) that  $|C_k(A)+C_k(B)|\leq |A+B|$ for any finite $A,B\subset V$. We say that $A$ is compressed if $C_k(A)=A$ for all $k$. A compressed set $A\subset V_d$ has the property that if $(a_1,\ldots,a_d)\in A$ and $b_i\in \ZZ$ with $0\leq b_i\leq a_i$ for all $1 \le i \le d$, then $(b_1,\ldots,b_d)\in A$. The next lemma will allow us to assume that $A$ is both compressed and of low dimension when proving our main result.

\begin{lem} \label{lem:lowdim}
Suppose that $A\subset V$ is finite with $|A+\Phi(A)|=K|A|$. Then there is some $d\leq 2K$ and $A'\subset V_d$ with $|A'|=|A|$ such that $A'$ is compressed and $|A'+\Phi(A')|\leq |A+\Phi(A)|$.
\end{lem}

\begin{proof}
Since $A$ is finite, $A\subset V_D$ for some $D$. Note that $\Phi\circ C_i=C_{i+1}\circ \Phi$ for all $i$. Denote by $C_{[i]}$ the composition $C_1\circ C_2\circ\cdots \circ C_i$. Then $C_{[D+1]}(A)=C_{[D]}(A)$ and $C_{[D+1]}(\Phi(A))=\Phi(C_{[D]}(A))$. Thus, setting $A_1=C_{[D]}(A)$, we have $|A_1|=|A|$ and 
\begin{align*}
    |A_1+\Phi(A_1)| &= |C_{[D]}(A)+\Phi(C_{[D]}(A))| = |C_{[D+1]}(A)+C_{[D+1]}(\Phi(A))| \leq  |A+\Phi(A)|.
\end{align*}
Furthermore, $A_1$ is compressed. 

Let $e_k=\lambda^{k-1}$ be the basis vectors for $k=1,\ldots,D$. If $e_k\not\in A_1$, then the $k$-th coordinate of every point of $A_1$ is 0. Let $A_1'$ be the set formed by replacing each point $(x_1,\ldots,x_{k-1},0,x_k,\ldots,x_{D-1})$ of $A_1$ with the point $(x_1,\ldots,x_{k-1},x_k,\ldots,x_{D-1})$, so that $A_1'\subset V_{D-1}$. We claim that $|A_1'+\Phi(A_1')|\leq |A_1+\Phi(A_1)|$. Indeed, every point of $A_1+\Phi(A_1)$ is of the form 
$$(x_1,x_2+y_1,x_3+y_2,\ldots,x_{k-1}+y_{k-2},y_{k-1},x_k,x_{k+1}+y_k,\ldots,x_{D-1}+y_{D-2},y_{D-1})$$
for some $(x_1,\ldots,x_{k-1},0,x_k,\ldots,x_{D-1}),(y_1,\ldots,y_{k-1},0,y_k,\ldots,y_{D-1})\in A_1$, whereas every point of $A_1'+\Phi(A_1')$ is of the form
$$(x_1,x_2+y_1,x_3+y_2,\ldots,x_{D-1}+y_{D-2},y_{D-1}).$$
There is a clear surjection from $A_1+\Phi(A_1)$ to $A_1'+\Phi(A_1')$ by summing and combining the $k$-th and $(k+1)$-th coordinates.

Repeating the above procedure whenever possible for each $k$, we obtain a set $A'$ with $|A'| = |A|$, $|A'+\Phi(A')|\leq |A+\Phi(A)|$ and $A'\subset V_d$ for some $d$ with $e_k\in A'$ for $k=1,\ldots,d$. By this last condition, $A'$ is $d$-dimensional and, moreover, $A'+\Phi(A')$ is $(d+1)$-dimensional. Hence, by Lemma~\ref{lem:hdsums}, we have $|A'+\Phi(A')|\geq (d+2)|A'|-\frac{(d+1)(d+2)}{2}$. Using that $|A'+\Phi(A')|\le K|A'|$ and $|A'|\geq d+1$, we get $d\leq 2K$, as required.
\end{proof}

We also note the following result of Pl\"unnecke--Ruzsa type.

\begin{lem} \label{lem:pr}
Suppose $A\subset V$ is finite. If $|A+\Phi(A)|\leq K|A|$ for some $K>0$, then $|(A+\Phi(A))+\Phi(A+\Phi(A))|\leq K^{10}|A|$.
\end{lem}

\begin{proof}
The sum version of Ruzsa's triangle inequality~\cite{R96} states that for any finite subsets $X,Y,Z$ of an abelian group,
$$|X||Y+Z|\leq |X+Y||X+Z|.$$
Setting $X=\Phi(A)$, $Y=Z=A$ and noting that $|\Phi(A)|=|A|$, we have
$$|\Phi(A)||A+A|\leq |A+\Phi(A)||A+\Phi(A)|,$$
so that $|A+A|\leq K^2|A|$. Hence, by the Pl\"unnecke--Ruzsa inequality, $|A+A+A+A|\leq K^8|A|$. Thus, another application of Ruzsa's triangle inequality (with $X=\Phi(A)$, $Y=A$, $Z=\Phi(A)+\Phi(A)+\Phi(A)$) yields
$$|\Phi(A)||A+\Phi(A)+\Phi(A)+\Phi(A)|\leq |A+\Phi(A)||\Phi(A)+\Phi(A)+\Phi(A)+\Phi(A)|,$$
so that $|A+\Phi(A)+\Phi(A)+\Phi(A)|\leq K^9|A|$. Applying Ruzsa's triangle inequality once more (with $X=\Phi(A)$, $Y=A+\Phi(A)+\Phi(A)$, $Z=\Phi^2(A)$), we see that
$$|\Phi(A)||A+\Phi(A)+\Phi(A)+\Phi^2(A)|\leq |A+\Phi(A)+\Phi(A)+\Phi(A)||\Phi(A)+\Phi^2(A)|,$$
so that $|A+\Phi(A)+\Phi(A)+\Phi^2(A)|\leq K^{10}|A|$, as required.
\end{proof}

We now come to the main novel ingredient in our proof, which is a strong upper bound for the size of the projections of any compressed $A\subset V_d$ in terms of $|A+\Phi(A)|$. 
Given a set $I\subseteq [d]$, we will write $\alpha(I)$ for the length of the longest set of consecutive integers in $I$. 

\begin{lem} \label{lem:noplane}
Let $A\subset V_d$ be finite and compressed with $|A+\Phi(A)|=N$. Then, for any subset $I\subseteq [d]$,
$$|p_I(A)|\leq N^{\frac{k}{k+1}},$$
where $k=\alpha(I)$.
\end{lem}

\begin{proof}
For any set of integers $J$, define $\phi(J)=\setcond{j+1}{j\in J}$. We claim that, for any $J_1,J_2\subset [d]$,
$$\frac{|p_{J_1}(A)||p_{J_2}(A)|}{|p_{J_1\cap \phi(J_2)}(A)|}\leq N.$$
To show this, we will exhibit an injection $p_{J_1}(A)\times p_{J_2}(A)\to p_{J_1\cap \phi(J_2)}(A)\times (A+\Phi(A))$. Let $(x,y)\in p_{J_1}(A)\times p_{J_2}(A)$ and 
consider the map
$$(x,y)\mapsto (p_{J_1\cap \phi(J_2)}(x), x+\Phi(y)).$$
Since $A$ is compressed, $p_J(A)\subseteq A$ for every $J$, which easily implies that $(p_{J_1\cap \phi(J_2)}(x), x+\Phi(y))$ is indeed in $p_{J_1\cap \phi(J_2)}(A)\times (A+\Phi(A))$. 
To see that the map is injective, it is enough to observe that 
$$x=p_{J_1\cap \phi(J_2)}(x)+p_{J_1\setminus \phi(J_2)}(x)=p_{J_1\cap \phi(J_2)}(x)+p_{J_1\setminus \phi(J_2)}(x+\Phi(y))$$
and 
$$\Phi(y)=p_{\phi(J_2)}(\Phi(y))=p_{\phi(J_2)}(x+\Phi(y))-p_{\phi(J_2)}(x)=p_{\phi(J_2)}(x+\Phi(y))-p_{J_1\cap \phi(J_2)}(x).$$

For $i=0,1,\ldots,k$, let
$$I_i=\setcond{j\in I}{\set{j,j-1,\ldots,j-i}\subseteq I}.$$
Then $I=I_0\supset I_1\supset\cdots \supset I_{k}=\emptyset$ 
and, for each $i=0,1,\ldots,k-1$, $I\cap \phi(I_i)=I_{i+1}$. Thus, by the claim above, 
$$\frac{|p_I(A)||p_{I_i}(A)|}{|p_{I_{i+1}}(A)|}\leq N.$$
Taking the product of this inequality over all $i=0,1\ldots,k-1$, we get
$$|p_I(A)|^{k+1}\leq N^{k}$$
and the lemma follows.
\end{proof}

We are now ready to prove our main result.

\begin{proof}[Proof of Theorem~\ref{thm:main}]
Suppose instead that $|A+\Phi(A)|=Kn$, where $K<e^{c\sqrt{\log n}}$ for some $c>0$ that will be fixed later. By Lemma~\ref{lem:lowdim}, we may assume that $A$ is compressed and $A \subset V_d$ with $d\leq 2K$.

By Lemma~\ref{lem:pr}, we have
$$|A+\Phi(A)+\Phi(A+\Phi(A))|\leq K^{10}n.$$
Since $A$ is compressed, so are $\Phi(A)$ and, therefore, $A+\Phi(A)$. Hence, Lemma~\ref{lem:noplane} implies that
$$|p_I(A+\Phi(A))|\leq (K^{10}n)^{\frac{k}{k+1}}$$
for any $I\subseteq [d+1]$, where $k=\alpha(I)$. 
But the number of $I\subseteq [d+1]$ with $\alpha(I)=k$ is at most
$$\sum_{i=1}^{d+2-k}\abs{\setcond{I\subseteq [d+1]}{i,i+1,\ldots,i+k-1\in I}}\leq (d+2)2^{d+1-k}.$$
Thus, by Lemma~\ref{lem:discbm}, we have that
\begin{align*}
    2^{d+1}n &\leq \sum_{I\subseteq [d+1]} |p_I(A+\Phi(A))| \leq \sum_{k=0}^{d+1}\abs{\setcond{I\subseteq [d+1]}{\alpha(I)=k} }(K^{10}n)^{\frac{k}{k+1}}\\
    &\leq \sum_{k=0}^{d+1} (d+2)2^{d+1-k}(K^{10}n)^{\frac{k}{k+1}}.
\end{align*}
Therefore, 
\begin{align*}
1 &\leq \sum_{k=0}^{d+1} (d+2)2^{-k}K^{\frac{10k}{k+1}}n^{-\frac{1}{k+1}} \leq 2(d+2)\sum_{k=0}^{d+1} 2^{-k-1}K^{10}n^{-\frac{1}{k+1}}\\
&\leq 2(d+2)\sum_{k=0}^{d+1} e^{-(k+1)\log 2+10c\sqrt{\log n}-\frac{\log n}{k+1}}\\
&\leq 2(d+2)\sum_{k=0}^{d+1} e^{-2\sqrt{(\log 2)\log n}+10c\sqrt{\log n}}\quad \paren{\text{using $(k+1)\log 2+\frac{\log n}{k+1}\geq 2\sqrt{(\log 2)\log n}$}}\\
&= 2(d+2)^2e^{(10c-2\sqrt{\log 2})\sqrt{\log n}} \leq e^{(13c-2\sqrt{\log 2})\sqrt{\log n}},
\end{align*}
which is a contradiction for $c=0.1$ and $n$ sufficiently large. For smaller $n$, we may use the trivial estimate $|A+\Phi(A)| \geq 2|A|-1$ to choose an appropriate $c$ that works for all $n$.
\end{proof}

As a final remark, we note that the conclusion of Theorem~\ref{thm:mainintro} also holds for any finite subset $A$ of $\mathbb{C}$ and any transcendental $\lambda \in \mathbb{C}$. Indeed, Lemma~\ref{lem:alg} again reduces the problem to estimating $|A + \lambda \cdot A|$ for finite $A \subset \mathbb{Q}[\lambda]$ and then to estimating $|A + \Phi(A)|$ for finite $A \subset V$, so the rest of the proof goes through without change.

\end{document}